\documentclass[10pt]{article}
\usepackage[a4paper, total={6in, 8in}]{geometry}
\usepackage{mathrsfs,amsfonts,amsmath,amssymb, mathtools, amsthm}
\usepackage{latexsym, xcolor}

\newtheorem{lemma}{Lemma}
\newtheorem{theorem}{Theorem}
\newtheorem{proposition}{Proposition}
\newtheorem{conjecture}{Conjecture}

\renewcommand{\thefootnote}{}

\def\R{\mathbb {R}}
\def\E{\mathsf{E}}

\def\C{\mathbb{C}}

\def\({\big (}
\def\){\big )}

\def\eps{\varepsilon}
\def \ES{Erd{\H o}s-Szemer\'edi~}
\def \ST{Szemer\'edi-Trotter~}
\def \erdos{Erd\H os~}
\def \szem{Szemer\'edi~}
\def\exponent{\frac{2}{1167}}
\def \Reps{\mathcal{R}_\eps}

\newcommand{\Addresses}{{
  \bigskip
  \footnotesize

  M.~Rudnev, \texttt{Department of Mathematics, University of Bristol, UK}\par\nopagebreak
  \textit{E-mail address:} \texttt{misha.rudnev@bristol.ac.uk}

  \medskip

  S.~Stevens, \texttt{Johann Radon Institute for Computational and Applied Mathematics (RICAM), Linz, Austria}\par\nopagebreak
  \textit{E-mail address:}, \texttt{sophie.stevens@ricam.oeaw.ac.at}

}}

\title{An update on the sum-product problem}
\author{Misha Rudnev \thanks{The first author is partially supported by the Leverhulme  Trust Grant RPG--2017--371.}\\ University of Bristol, UK 
\and\  Sophie Stevens \thanks{The second author is supported by the Austrian Science Fund FWF Project P~30405-N32. }\\ RICAM, Linz, Austria}

\begin{document}

\maketitle
\abstract{
We improve the best known sum-product estimates over the reals.  We prove that
\[
\max(|A+A|,|A+A|)\geq |A|^{\frac{4}{3} + \frac{2}{1167} - o(1)}\,,
\]
for a  finite $A\subset \R$, following a streamlining of the arguments of Solymosi, Konyagin and Shkredov. We include several new observations to our techniques.

Furthermore,
\[
|AA+AA|\geq  |A|^{\frac{127}{80} - o(1)}\,.
\]
Besides, for a convex set $A$ we show that
\[
|A+A|\geq |A|^{\frac{30}{19}-o(1)}\,.
\]
This paper is largely self-contained.
}

\section{Introduction}\label{sec:intro}
Throughout $A\subset \R$ is a finite set of positive real numbers, whose cardinality $|A|$ exceeds some absolute constant. All other sets, 
denoted by upper-case letters, are also finite. The sumset of two sets $A,B$ is defined as
\[
A+B:=\{a+b\colon a\in A,\,b\in B\}\,,
\]
with similar notations for the product and ratio sets, $AB$, $A/B$, etc. 

The \ES sum-product conjecture, originally stated over the integers \cite{ES}, is the following.	
	\begin{conjecture}	\label{conj:sumprod}
		For all $\delta < 1$ and sufficiently large $A\subseteq \R$  ,
		\begin{equation}\label{e:sumprodconj}
			\max\{|A+A|,|AA|\}\geq |A|^{1+\delta}\,.
		\end{equation}	 
	\end{conjecture}
	
As shorthand notation, already used in the abstract, we could instead write this as $\max\{|A+A|,|AA|\}\geq |A|^{2 -o(1)}$.

Historically, \erdos and \szem proved \cite{ES} a qualitative (but quantifiable) sum-product estimate \eqref{e:sumprodconj} with (in the notation of Conjecture~\ref{conj:sumprod}) some $\delta >0$, which Nathanson \cite{nathanson} and Ford \cite{Ford} showed to be $\delta = 1/31$ and $\delta = 1/15$ respectively. The textbook proof by Elekes \cite{Elekes} uses the geometric \ST theorem to advance to $\delta = 1/4$.  In 2008 Solymosi \cite{Solymosi} developed a different type of geometric argument to prove $\delta = 1/3-o(1)$, which stood until Konyagin and Shkredov  \cite{KS1, KS2} developed a synthetic approach, which enabled them to pass slightly beyond the value $\delta = 1/3$. Owing to technical  improvements in  \cite{RSS, Shakan}, the latter paper \cite{Shakan} by Shakan holds the current world record $\delta  = \frac{1}{3}+\frac{5}{5277}-o(1)$. We push these technical developments yet further, and streamline the arguments in this approach, aiming to identify where it may be subject to improvement.
	
	\begin{theorem}\label{t:sumprod}
	For a finite $A\subseteq \R$ one has
	\begin{equation}
	\label{eq:sumprod}
	\max\{|AA|,|A+A|\}\geq |A|^{\frac{4}{3}+\exponent -o(1)}\,.
	\end{equation} 
	\end{theorem}

As a by-product of the techniques used in this paper we also prove a new bound on the cardinality of the set $AA+AA$. Geometrically, this is the set of dot products of pairs of vectors in the point set $A\times A\subset \R^2$. We note the the number of distinct dot products generated by a general set of points $\mathcal{P}\subseteq \R^2$, no better lower bound than $|\mathcal{P}|^{2/3}$ is known; this lower bound comes from a single application of the \ST theorem (or, in fact, a single application of the weaker Beck theorem), see \cite{IRR} for relevant discussion.

	\begin{theorem}\label{t:aaaa}
For a finite $A\subseteq \R$ one has
\[
|AA+AA|\geq |A|^{\frac{3}{2} + \frac{7}{80}-o(1)}\,.
\]
\end{theorem}
This improves on the previously best known exponent $\frac{3}{2}+ \frac{1}{12} - o(1)$ by Iosevich, Roche-Newton and the first author \cite{IRR}.
	
Another implication of our techniques is a new bound on the number of convex sums. The real set $A=\{a_1< a_2<\dots <a_n\}$ is convex if the sequence of consecutive differences $a_{i+1}-a_i, \, i = 1,\dots, n-1$ is strictly increasing.  Without loss of generality, we have $a_i = f(i)$, for some strictly convex real smooth function $f(x)$, so we write $A=f([n])$, where $[n]:=\{1,\ldots,n\}$.

Schoen and Shkredov \cite{schoenshkredov} showed that a convex set $A$ satisfies $|A+A|\geq |A|^{14/9-o(1)}$; we improve the current best exponent $\frac{102}{65} -o(1)$ due to Olmezov \cite{Olmezov}.

\begin{theorem}\label{t:convex}
Let $A\subseteq \R$ be convex. Then
\[
|A+A|\geq  |A|^{30/19-o(1)}\,.
\]
\end{theorem}

The best known exponent for the set of differences $A-A$, also due to Schoen and Shkredov \cite{schoenshkredov} is slightly better: $|A-A|\geq |A|^{8/5-o(1)}$. 
Formalising the heuristic that convexity should destroy additive structure, \erdos \cite{erdos2} conjectured the lower bound $|A\pm A|\geq |A|^{2-o(1)}$ for convex $A$.
Stronger bounds $|A\pm A|\geq |A|^{5/3-o(1)}$ were recently proven by Olmezov \cite{Ol} under additional assumptions on higher derivatives of $f$.

\subsubsection*{Organisation of paper}
The paper is arranged as follows. 
In Section~\ref{sec:discussion} we present  a non-technical outline of the proof of Theorem \ref{t:sumprod}. We follow the strategy of Konyagin and Shkredov \cite{KS1}. We owe much of our quantitative improvement by applying the forthcoming Theorem~\ref{t:main}, replacing a prototype result due to Shkredov \cite{ShkredovSums}. The remainder of our quantitative improvement is summarised in Proposition~\ref{p:main}, the proof of which we discuss in Section~\ref{sec:discussion2}.

In Section~\ref{sec:prelim} we record standard results about energy estimates together with their proofs to keep the exposition self-contained and maximally jargon-free\footnote{{\em I have always thought that clarity is a form of courtesy that the philosopher owes\dots This is different from the individual sciences which increasingly interpose between the treasure of their discoveries and the curiosity of the profane the tremendous dragon of ... terminology. } J. Ortega y Gasset, 1929, see \cite{Ort}.}.

In Section~\ref{sec:prop} we state and prove Proposition~\ref{p:main}. Similar results were proven by Konyagin and Shkredov \cite{KS1,KS2} and Shakan \cite{Shakan}. In contrast to previous literature, we study a slightly different quantity and our proof is entirely elementary.

Section~\ref{sec:proofofsumprod}contains the proof of our main sum-product result, Theorem~\ref{t:sumprod}. 

In Section~\ref{sec:proofofthm2}, we introduce a regularisation lemma, Lemma \ref{lem:regu}, which was  originally proved in \cite[Lemma 3.1]{RShakanS}; it had a more cumbersome precursor in the form of  \cite[Lemma 7.2]{MRSS}.  Lemma \ref{lem:regu} replaces the use of Shkredov's {\em spectral method} (see the foundational paper \cite{shkredov} and references therein
    ) that is typically used in this kind of argument.  Statements reminiscent of Lemma \ref{lem:regu} may be of broader interest in terms of various energy variants of the so-called Balog-Wooley decomposition \cite{BW}, used in \cite{RSS, Shakan,ShkredovRemarks}, where they would streamline lengthier arguments. We prove Theorem~\ref{t:main} and Theorem~\ref{t:convex}. 
    
Finally, Section~\ref{sec:aaaa} contains the proof of Theorem~\ref{t:aaaa}. It is only by using the argument within the proof of Theorem~\ref{t:main} that we are able to claim a new lower bound.

\renewcommand{\thefootnote}{\arabic{footnote}}

\subsubsection*{Notation}
The symbols $\ll,\gg,\sim$ suppress absolute constants in inequalities; $\lesssim,\,\gtrsim$ also  suppress powers of $\log|A|$. The Vinogradov symbol $\ll$ may acquire a subscript, say $\ll_s$ to indicate that the hidden constant depends on a parameter $s$.

We write $aA=Aa:=\{a\}A$ to denote the dilate of $A$ by $a\neq 0$, similarly $A+a$ or $a+A$ for a translate. We will use the \emph{number of realisations} notation $r_{A+B}(x):=|\{(a,b)\in A\times B\colon x = a+b\}|$ to denote the number of realisations of the number $x$ as a sum of elements from sets $A$ and $B$. Similar notation will be used for e.g. the number of realisations of $x$ as an element of $AA$, as $r_{AA}(x)$, etc.
	
\subsection{Outline of the proof of Theorem~\ref{t:sumprod}}\label{sec:discussion}
Solymosi's renowned result \cite{Solymosi} related the sumset of a set of positive reals $A$ to its multiplicative energy 
\[\E^\times(A):= \left|\left\{(a,b,c,d)\in A^4\colon \frac{a}{b} = \frac{c}{d}\right\}\right| = \left|\left\{(a,b,c,d)\in A^4\colon {a}{d} = bc\right\}\right| \,.\]
Geometrically, $\E^\times(A)$ is the sum of the numbers of pairs of points of $A\times A$ supported on lines through the origin, with slopes in $A/A$. 

To sketch Solymosi's well-known argument as a base for further build-up, consider the model case: suppose that each line through the origin with slope $\lambda\in A/A$ supports the same number of points $\tau$ of $A\times A$. i.e. in the \emph{number of realisations} notation, this is synonymous with $\forall \lambda\in A/A,\,r_{A/A}(\lambda)=\tau$. Then $\E^\times(A) = \tau^2 |A/A|$. 

Solymosi observed that taking vector sums of all pairs of points lying on pairs of lines with consecutive slopes yields distinct elements of $(A+A)\times (A+A)$: hence $|A+A|^2\geq \tau^2 |A/A| = \E^\times(A)$. 

The restrictive assumption that each $\lambda \in A/A$ has the same number number of realisations $\tau$ is dismissed  via the standard dyadic pigeonholing argument. This slightly weakens the above, to what we will refer to as {\em Solymosi's inequality}: 

\begin{equation}\label{e:sol}
\E^\times(A)\leq 4|A+A|^2\lceil \log|A|\rceil\,.
\end{equation}
If $A=[n]$, the inequality is sharp up to a constant. It follows by the Cauchy-Schwarz inequality that
\begin{equation}\label{e:csmult}
|A+A|^2 |AA| \gg |A|^4 \log^{-1}|A|\,,
\end{equation}
hence the sum-product exponent $\delta=\frac{1}{3}-o(1)$, in terms of \eqref{e:sumprodconj}. The power of $\log|A|$, hidden in the $o(1)$ term is no longer sharp in the case $A=[n]$, but nonetheless the $o(1)$ term remains, see  \cite{Ford} for its precise asymptotics.

However, if $A=[n]$, then $|AA|\geq |A|^{2-o(1)},$ and in terms of Conjecture \ref{conj:sumprod} there is nothing to prove. Assuming that $A$ has more multiplicative structure, ie. $\E^\times(A) \gg |A|^{2 + o(1)}$, Konyagin and Shkredov designed the following procedure. Let us once again assume that for every $\lambda\in A/A,\,r_{A/A}(\lambda)=\tau$. Set
\[
N:= C\frac{|A+A|^2}{\E^\times(A)}\,,\]
for a suitably large $C$. The purpose is to estimate $N$ from below in the worst possible scenario for the sum-product inequality: when $|A+A|=|AA|$. If $N\gg|A|^{3\epsilon}$, one gets the improvement $\frac{1}{3}\to \frac{1}{3}+\epsilon$ to Solymosi's value of $\delta$ in the sum-product exponent.

Partition $A/A$ by consecutive \emph{bunches} (rather than pairs) of $N>2$  consecutive slopes. Suppose that each bunch of $N\tau$ points contributes $\gg N^2\tau^2$ distinct vector sums, rather than $N\tau^2$ as Solymosi's estimate counts. This contradicts the definition of $N$, leading to $C\ll 1$. Hence, there are many collisions between pairwise vector sums within a bunch and each vector sum is attained (as a sum of two points lying  on distinct slopes within the same bunch)  roughly $N$ times.

Interpreting this in terms of vector sums leads us to an algebraic conclusion: for most of the slopes $\lambda \in A/A$, there are $\gg \tau^2 N^{-2}$ solutions to the equation
$a = a_1 + a_2$
where $a\in A_\lambda:= A\cap \lambda^{-1} A$ (i.e.,  $a$ is an $x$-coordinate of a point in $A\times A$ lying on the line through the origin with slope $\lambda$), and $a_1,a_2$ lie respectively some dilates of some $A_{\lambda_1},\,A_{\lambda_2}$, with $\lambda_{1},\lambda_{2}$ coming from the same bunch as $\lambda$. By construction, each variable $a,a_1,a_2$ runs through $\tau$ values, so that the number of solutions to the above equation is nearly maximum possible when $N$ is small.

Upon this conclusion Konyagin and Shkredov took advantage of a maxim of Elekes and Ruzsa \cite{ER}, that {\em few sums imply many products}. This required generalising the approach of \cite{ER} and was achieved with increasing efficiency in \cite{KS1, KS2, RSS, Shakan}. Here we present a lucid and self-contained version of the argument, aiming at minimum auxiliary notation. The analysis is presented within the proof of forthcoming Proposition \ref{p:main}, whose key conclusion is that, under the scenario in consideration, sets $A_\lambda$ must have small multiplicative energy.

This implies that the product sets $A_\lambda A_\lambda, AA_\lambda$ are quite large. Our somewhat different numerology allows us the additional new benefit of using the latter product set. By the truism that Konyagin and Shkredov call the Katz and Koester \cite{KK}  inclusion, $A A_\lambda$ being large means that $\lambda$ has at least $|AA_\lambda|$ realisations as a ratio from $AA/AA$. By slicing $A\times A$ with a vertical line, a subset of roughly $\gg |A|$ such slopes $\lambda$ can be identified with a subset of $A$, each member of which has many representations as a ratio from $AA/AA$. Shkredov \cite{ShkredovSums} called such sets \ST sets and proved that they have fairly large sumsets \cite[Theorem 11]{ShkredovSums}.  

We avoid following Shkredov's rather general line of notation apropos of the  \ST sets (see also e.g. \cite{ShkredovRemarks}, this notation was adopted by Shakan in \cite{Shakan}). Instead we spell out the suitable (and  stronger) estimate in Theorem \ref{t:main}, which yields a lower bound on $N$,  thus concluding the proof of Theorem \ref{t:sumprod}.

    \begin{theorem} \label{t:main} 
    Let $A,\Pi \subset \R\backslash\{0\}$ be finite, with  $|\Pi|\geq |A|$ and $r_{\Pi/\Pi}(a)\geq T$ for all $a\in A$, for some $T\geq 1$. 
    
    Then 
    \[
	    |A+A|^{19} |\Pi|^{44} \geq |A|^{41-o(1)} T^{33}\,.
    \]
    \end{theorem}
		

With $\Pi=AA$, the analogous inequality used by Konyagin, Shkredov and others, in  e.g. \cite{KS1, KS2, RSS, Shakan} was
	\[
	|A+A|\geq |A|^{58/37-o(1)}\cdot  \left(d_+(A):= \frac{|AA|^4}{T^3|A|}\right)^{-21/37}\;\;\Rightarrow\;\; |A+A|^{37}|AA|^{84} \geq |A|^{79-o(1)}T^{63}\,.
	\]
	The estimate for $\max\{|AA|,\,|A+A|\}$ given by Theorem \ref{t:main} is better whenever $|A|^{-16} T^{24} >1$. In the context of the implementation of the Konyagin-Shkredov strategy this is indeed the case, as one roughly has $T\sim |A|^{4/3}$.

\medskip
\subsection{Outline of the proof of Theorem \ref{t:main}}\label{sec:discussion2}
For finite subsets $A, B$ of an additive group and a real $s\geq 1$, define the $s$-th energy of $A$ and $B$:
	\[
		\E_s(A,B) := \sum_x r^s_{A-B}(x)\,,
	\]
	where  $r_{A-B}(x)$ is the number of realisations of the difference $x$. If $B=A,$ we write $\E_s(A)$ and if $s = 2$ we write $\E(A,B)$. In the special case $s=2$ we can replace instances of addition with subtraction. In a multiplicative group this definition coincides exactly with the multiplicative energy $\E^\times(A)$ defined above, the notation $\E$ for energy with respect to addition in $\R$ bears no superscript.

Energy, as the number of solutions of one equation of several variables, can be estimated from above via incidence theorems. For reals, this is first and foremost the (generally sharp) Szemer\'edi-Trotter theorem \cite{SzT}, bounding the number of incidences between a set of points and a set of lines (or curves that ``behave like lines") in $\R^2$. In this paper, the point set is always a Cartesian product and so the  Szemer\'edi-Trotter incidence bound can be derived using only elementary arguments. Namely, the order-based, elementary double counting ``lucky pairs'' argument, that we first encountered in the paper by Solymosi and Tardos \cite{ST}. 

For a set $A$, which is convex or is such that $|AA|\ll|A|$ and any set $B$, the energy  $\E_3(A,B)$ can be 
bounded as $|A|^{-1}$ times the number of collinear triples in the set $A\times B$, which the  \ST theorem bounds sharply (up to constants) as $|A|^2|B|^2 \log |A|$. In view of this, the third (or cubic) energy $\E_3(A,B)$  has played a key role in the strongest known sum-product type results over the reals, see e.g. \cite{MRSS,OSS}, as well as convex set bounds \cite{schoenshkredov, Ol}.

Consider the following truism on triples $(a,b,c)\in A^3$: \begin{equation}\label{e:truismintro}
	b-c=(a+b)-(a+c)\,.
	\end{equation}
This equation can be rewritten as $d= x-y$, where $d\in A-A$, $x,y\in A+A$. One defines an equivalence relation on triples $(a,b,c)$ yielding the same $(d,x,y)$: this happens if and only if one adds the same $t\in \R$ to $b,c$ and subtracts the  number $t$ from $a$. If $\sigma$ denotes an equivalence class containing $r(\sigma)$ triples $(a,b,c)\in A^3$, it is easy to bound
\[
\sum_\sigma r^2(\sigma) \leq \E_3(A)\,.
\]
There are $ |A|^3$ solutions $(a,b,c)$  to \eqref{e:truismintro}, yet we will later impose some restrictions on $b-c, a+b$. On the other hand, we can use the Cauchy-Schwarz inequality to relate the number of solutions to \eqref{e:truismintro} to the product of $\sum_\sigma r^2(\sigma)$ and the number of solutions to $d=x-y$.  We use the \ST theorem and the H\"older inequality to get an upper bound on the number of solutions of the  equation $d= x-y$.

This approach (founded in a series of works by Shkredov, see e.g. \cite{shkredov}) has been the key strategy in proving the recent {\em few products, many sums} results over the reals, towards the so-called {\em weak \ES conjecture} over $\R$. The weak \ES conjecture claims, roughly  that if $|AA| \to |A|$, then $|A+A| \to |A|^{2-o(1)}$, with the parameter dependence hidden in $\to$ being polynomial. This is a {\em few products, many sums} situation.  In contrast to the converse {\em few sums, many products} scenario by Elekes and Ruzsa, the former question is wide open, the best known results can be found in \cite{OSS}, and states that when $|AA|\to |A|$, then $|A+A|\to |A|^{8/5-o(1)}\,.$

The problem is that the equation  $d= x-y$, where $d\in A-A$, $x,y\in A+A$ involves unavoidably the differences from $A$, which can be generally related to sums only via  the additive energy  $\E(A)$. This forces one to restrict $d$ (as well as the quantities $x,y$ for the purpose of being able to prove good upper bounds) to some popular subsets of $A-A$ and $A+A$. This makes the set of $(a,b,c),$ on which the truism  \eqref{e:truismintro} is considered thinner, undermining the lower bound on the number of solutions of the equation $d= x-y$. 
	
Shkredov's spectral method, see e.g.  \cite{shkredov, ShkredovRemarks, MRSS, OSS} successfully provides lower bounds, involving restricted subsets of differences and sums,  by extending the equation $d=x-y$ to
	\begin{equation}\label{e:ks_fundamental}
		\alpha-\beta = d = x-y\,,
	\end{equation}
	with the additional variables $\alpha,\beta \in A$.
However, this creates additional challenges for proving upper bounds on the number of solutions. The key element of the proof of Theorem \ref{t:main} is the use of Lemma \ref{lem:regu} instead of the spectral method. This enables us to avoid  \eqref{e:ks_fundamental}, and deal instead with the equation $d=x-y$, where we can provide both suitable lower and upper bounds for the number of solutions, under the required popularity assumptions on the quantities $d,x,y$. 
However, we know no way to do without the spectral method for estimating $\E(A)$ for $A$ with small multiplicative doubling \cite{MRSS, OSS}.

\section{Preliminaries}\label{sec:prelim}
The lemmata in this section present the (elementary) version of the Szemer\'edi-Trotter theorem we need and show how it is used to furnish energy estimates.

\begin{lemma}\label{lem:stest}
Let $A,B\subset \R$ be finite sets and $L$ a set of affine lines or translates of a strictly convex curve $y=f(x)$. The number of incidences between the point set $A\times B$ with $L$ is\footnote{The term $|L|$ in this estimate can be written more precisely as $|\{l\in L:\,|l\cap (A\times B)|=1\}|.$}  
$~{\displaystyle O\left( (|A||B||L|)^{\frac{2}{3}}+|L| \right).}$

In particular, for $k\geq 2$ the number of lines (curves) with $\geq k$ points is ${\displaystyle O\left(\frac{(|A||B|)^2}{k^3}\right)\,.}$
 \end{lemma} 

Note that affine lines have finite nonzero slopes.  For an elementary ``lucky pairs'' proof  of Lemma \ref{lem:stest} when $L$ is the set of affine lines see \cite{ST}. The same proof works for translates of a convex curve.

The next two lemmata collect the bounds we need, based on Lemma \ref{lem:stest}.

\begin{lemma}\label{lem:estimatinge3}
Let $A, B, C,\, \Pi_1,\Pi_2\subset \R$ be finite sets with  the property that $r_{\Pi_1\Pi_2}(a)\geq T$ for all $a\in A$ and some $T\geq 1$.
Then if $|\Pi_1||C|\leq |\Pi_2|^2|B|^2,$
\begin{equation} \label{e:bdone}
|\{c=a-b;\,a\in A, b\in B,\,c\in C\}| \ll \frac{(|\Pi_1||\Pi_2||B||C|)^{2/3}}{T}\,.
\end{equation}
Besides, if $|\Pi_1||A|\leq |\Pi_2|^2|B|,$
\begin{equation}\label{e:e3(b,x)}
\E_3(A,B)\ll\frac{|B|^2|\Pi_1|^2|\Pi_2|^2\log|A|}{T^3}
\end{equation}
and for $s \in (1,3)$
\begin{equation}\label{e:e3/2(b,x)}
\E_s (A,B)\ll_s\frac{(|\Pi_1| |\Pi_2|)^{s-1}|B|^{\frac{s+1}{2}}|A|^{\frac{3-s}{2}}}{T^{\frac{3(s-1)}{2}}}\,.
\end{equation}

Furthermore, if $A$ is convex, then if $|C|\leq |A||B|^2$,
$$
|\{c=a-b;\,a\in A, b\in B,\,c\in C\}| \ll |A|^{1/3}(||B||C|)^{2/3}\,.
$$ and
for $s \in (1,3)$
\[
\E_s(A,B)\ll_s |A||B|^{\frac{s+1}{2}}\,, \qquad \E_3(A,B)\ll |A||B|^2\log|A| \,.
\]
\end{lemma}
\begin{proof} Note that the cardinality relations between the sets involved serve only for one to be able to disregard the trivial term $|L|$ in the \ST incidence estimate of Lemma  \ref{lem:stest}. Without loss of generality $|\Pi_1|\leq |\Pi_2|$.

To prove \eqref{e:bdone} observe that 
\begin{equation}\label{e:long} \begin{aligned}
|\{c=a-b;\,a\in A, b\in B,c\in C\}| & \leq T^{-1} |\{c=pq-b;\,p\in \Pi_1,q\in \Pi_2, b\in B, c\in C\}| \\ & \ll
T^{-1}[(|C||B||\Pi_1||\Pi_2|)^{2/3}  + |\Pi_1||C|] \\  & \ll T^{-1} (|C||B||\Pi_1||\Pi_2|)^{2/3} ,\,
\end{aligned}
\end{equation}
by Lemma \ref{lem:stest} and the assumptions on set cardinalities.

Furthermore, for an integer $k \in [1, \min(|A|,|B|)]$ define
\[
D_k :=\{d \in A-B:\,r_{A-B}(x)\geq k\}\,
\]
as the set of $k$-popular differences, clearly $|D_k|\leq |A||B|$.

By definition of $D_k$, as in \eqref{e:long}, with $D_k$ as $C$ in its right-hand side we have
\begin{align*}
|D_k|k  &\leq \frac1{T}|\{d=pq-b\colon d\in D_k,p\in \Pi_1,q\in \Pi_2,b\in B\}|\\
&\ll\frac1{T} (|D_k||B||\Pi_1||\Pi_2|)^{2/3}\,.
\end{align*}
Hence
\begin{equation} \label{e:dk}
|D_k| \ll \frac{ (|B||\Pi_1||\Pi_2|)^{2}}{(kT)^3}\,.
\end{equation}

Then bound \eqref{e:e3(b,x)} follows after dyadic summation in $k=2^j$, namely 
$$
\E_3(A,B) \ll \sum_{1\leq j \ll \log|A|} (2^{3j}) |D_{2^j}| \ll \frac{|B|^2|\Pi_1|^2|\Pi_2|^2\log|A|}{T^3},
$$
as claimed.

Similarly, for $1<s<3$ dyadic summation leads, for any $k \in [1,\min(|A|,|B|)]$, to the bound
$$
\sum_{x\in A-B:\,r_{A-B} (x) \geq k} r^s_{A-B}(x) \ll_s \frac{1}{k^{3-s}} \frac{|B|^2|\Pi_1|^2|\Pi_2|^2}{T^3}\,,
$$
where the hidden constant depends on $s$.

The remaining counterpart of $\E_s(A,B)$ is
$$\sum_{x\in A-B:\,r_{A-B} (x) < k} r^s_{A-B} (x) \leq k^{s-1} |A||B|\,.$$
Optimising the two latter bounds by choosing
$$
k = |\Pi_1| |\Pi_2| \sqrt{ \frac{|B|}{|A|T^3}}
$$ 
completes the proof of inequality \eqref{e:e3/2(b,x)}.

\medskip
For a convex $A=f([|A|])$ we want to show that same bounds as \eqref{e:e3(b,x)}, \eqref{e:e3/2(b,x)}  hold if one formally replaces $|\Pi_1|=|\Pi_2|=T=|A|$. Let us use the same notation $D_k$ as above, writing for a single representation of each of its element $d$ in $\geq \lfloor |A|/2\rfloor$ ways
\[
d = f(i) - b_i = f(i+j - j) - b_i = f(l - j) - b_i\,.
\]
(Without loss of generality we have assumed $i\leq |A|/2$: otherwise we wold do $i=(i-j) + j$, which leads to the same estimate using Lemma \ref{lem:stest}.)
Hence, $k|D_k|$ is bounded from above via $|A|^{-1}$ times the number of incidences between the point set $[|A|]\times B$ and $|A||D_k|$ translates of the curve $y=f(x)$. Applying Lemma \ref{lem:stest} yields, with no constraints on $B$
$$
|D_k|\ll \frac{|A||B|^2}{k^3}\,.
$$ The claimed energy  bounds follow similar to  \eqref{e:e3(b,x)}, \eqref{e:e3/2(b,x)}.
\end{proof}

Let us finally include the aforementioned {\em few sums, many products} estimate. The following lemma is in essence a restatement of \cite[Lemma 2.5]{VarI}. We provide a slightly shorter proof, which easily generalises for point sets that are not Cartesian products -- see the remark following the proof.
\begin{lemma}\label{l:fpms}
Let $A\subset \R\setminus\{0\}$ and $\Pi_1,\Pi_2\subset \R$ be finite sets with $|\Pi_2|,\, |\Pi_1|\, \geq \, |A|,$ and the property that $r_{\Pi_1- \Pi_2}(a)\geq T$ for all $a\in A$ and some $T\geq 1$.

Then
\begin{equation}\label{e:e2}
\E^\times (A) \ll \frac{ |\Pi_1|^{3}|\Pi_2|^{3} \log|\Pi_1| }{T^4}\,.
\end{equation}
\end{lemma}

\begin{proof}
To estimate 
$$\E^\times(A) = \left| \left\{ (a,b,c,d) \in A^4:\; \frac{a}{b}=\frac{c}{d}\right\}\right|$$ observe that 
this quantity is bounded from above by $T^{-4}$ times the number of solutions of the equation
\begin{equation}\label{slopes}
t = \frac{x-y}{x'-y'} = \frac{u-v}{u'-v'}\,:\,x,x',u,u'\in \Pi_1, \;y,y',v,v' \in \Pi_2\,,\;t\in \R\setminus\{0\}\,.
\end{equation}
The latter number of solutions, the variables $x,\ldots,v', t$ belonging to the sets as specified, is
$$
\sum_t  \left( \sum_{x,y'} r_{(x-\Pi_2)/(\Pi_1-y')} \right)^2 \leq |\Pi_1| |\Pi_2| \sum_t \sum_{x,y'} r^2_{(x-\Pi_2)/(\Pi_1-y')}(t)\,,
$$
by Cauchy-Schwarz. Rearranging the fractions 
\[
\sum_{t\in \R\setminus\{0\}} \sum_{x,y'} r^2_{(x-\Pi_2)/(\Pi_1-y')}(t)\ = \left|\left\{ \frac{x- v}{x- v'} =  \frac{y' - u}{y'-u'}:\, x,u,u'\in \Pi_1,\;y',v,v'\in \Pi_2\;: y'\neq u,u'\,, x\neq v,v'\right\}\right| \,.
\]

The latter quantity is the number of affine collinear triples with two $(u,v), (u',v')\in \Pi_1\times \Pi_2$ and one, different from both of the above,  $(y',x)\in \Pi_2\times \Pi_1$. The trivial aspect of the count when $(u,v) = (u',v')$ yields merely $|\Pi_1|^2|\Pi_2|^2$.

Otherwise by Lemma \ref{lem:stest}, the number of $k\geq2$ rich affine lines in $\Pi_1\times \Pi_2$ is $O\left(\frac{|\Pi_1|^2|\Pi_2|^2}{k^3}\right)$. We now take dyadic values $k_j=2^j,\,j\geq 1$, partitioning these lines into groups $L_{k_j}$, supporting the number of points of $\Pi_1\times \Pi_2$ in the  interval $[2^j,2^{j+1})$. The number of pairs of points from $\Pi_1\times \Pi_2$ on a line from the $j$th group is $\leq 4 k_j^2$. Furthermore, by Lemma \ref{lem:stest}, the number of incidences between $L_{k_j}$ and $\Pi_2\times\Pi_1$ is
$$ O\left( ( |L_{k_j}| |\Pi_1| |\Pi_2|)^{2/3} + |L_{k_j}|\right) \ll \frac{|\Pi_1|^{2} |\Pi_2|^{2}}{k_j^2}\,.$$
Multiplying by the latter bound by $ |\Pi_1||\Pi_2|k_j^2$ and summing in  $O(\log|\Pi_1|)$ values of $j$ absorbs the above trivial bound for the case $(u,v) = (u',v')$ and completes the proof.
\end{proof}
We remark that the proof of Lemma \ref{l:fpms}  easily adapts to  yield the following statement. 

\medskip{\em Let $P\subset \R^2$ have empty intersection with coordinate axes and $Q \subset \R^2$ meet any line $l$ in at most $|Q|^{1/2}$ points. Suppose, $\forall\,p\in P$, $r_{Q-Q}(p)\geq T^2$, for some $T\geq 1$ (or the same for $r_{Q+Q}(p)$). Then}
$$
\sum_{l :\,(0,0)\in l} |P\cap l|^2 \ll  \frac{ |Q|^{3}\log|Q| }{T^4}\,.
$$

\section{Proposition~\ref{p:main}}\label{sec:prop}
In this section we present our main proposition. Unlike the arguments of Konyagin and Shkredov \cite{KS1, KS2} and Shakan \cite{Shakan}, we lower bound the quantity $|A A_\lambda|$ instead of the smaller quantity $|A_\lambda A_\lambda|$ (definitions forthcoming). Our proof is elementary in that it uses only geometric and combinatorial observations. 

In contrast, Shakan's analogous proof used higher energies. One can apply Shakan's `less elementary' proof to the quantity $|AA_\lambda|$ to obtain the same conclusion as the forthcoming Proposition~\ref{p:main}. In fact, one can even replace Shakan's choice of second- and fourth-moment energies with the `more natural' (when using the Szemer\'edi-Trotter theorem) third moment energy to realise the same conclusion.

\begin{proposition}\label{p:main} 
Let $A\subseteq \R_{>0}$ and let $|A+A|=K|A|$ and $|AA|=M|A|$. Suppose that the multiplicative energy of $A$ is supported on slopes $S_\tau\subseteq A/A$ so that $r_{A/A}(\lambda)\in [\tau, 2\tau)$ for each $\lambda\in S_\tau$ and 
\[
|S_\tau|\tau^2 \leq \E^\times(A) \leq 4|S_\tau|\tau^2\log|A|\,.
\]

Then there exists $S_\tau'\subseteq S_\tau$ with $|S_\tau'|\geq\frac1{64} |S_\tau|$ 
such that, for every $\lambda \in S_\tau'$, 
\[
	|AA_\lambda| \gg \frac{|A|^6}{ M^4K^8|S_\tau|^{1/2}(\log|A|)^7}\,,
\]
where $A_\lambda = A \cap \lambda^{-1} A$.
 \end{proposition}
\begin{proof}
Take a natural number
\begin{equation}\label{e:en}
N := C K^2M |A|^{-1} \log|A|\,,
\end{equation}
for a sufficiently large $C$, say $C>128$. From Solymosi's inequality \eqref{e:sol}, it follows that $K^2 M \geq \frac1{4}|A|\lceil \log |A|\rceil^{-1}$, hence, in particular, $N>2$. 
We may also assume that $N<|A|^{1/2}$; indeed, if $N>|A|^{1/2}$, then it is enough to show that $|AA_\lambda|>C^4 |S_\tau|^{-1/2}\log(|A|)^{-3}$, which is readily satisfied since $|AA_\lambda|>|A|$.  

The slopes in $S_\tau$ are positive; we order them by increasing value and partition $S_\tau$ into bunches of $N$ consecutive lines. There are $\lfloor|S_\tau|/N\rfloor$ `full' bunches of exactly $N$ slopes, and at most one additional bunch with fewer than $N$ slopes that we delete with no consequence,  since $N$ is very small relative to $|S_\tau|$. Henceforth, we assume that all bunches are full.

Let $\ell_\lambda$ denote the line through the origin with slope $\lambda$.
For each pair of distinct slopes $\lambda_i,\lambda_j$ in a fixed bunch $\mathcal B$, we create between $\tau^2$ and $4\tau^2$ vectors in $(A+A)\times (A+A)$ from the sum of each of the $\sim \tau$ elements of $A\times A$ supported on the line $\ell_{\lambda_i}$,  with each of the $\sim \tau$ elements of $A\times A$ supported on $\ell_{\lambda_j}$. We denote these vectors by $\mathcal{A}_{\lambda_i} + \mathcal{A}_{\lambda_j}$.
These $\sim \tau^2$ vector sums lie between $\ell_{\lambda_i}$ and $\ell_{\lambda_j}$, and in particular, between the two extremal slopes in the bunch. A new element in $(A+A)\times(A+A)$ created in this manner could appear from multiple pairs of slopes within the same bunch; we must account for this over-counting. Note however that an element of $(A+A)\times (A+A)$ created in this way cannot have come from from pairs in two different bunches.  

Define, for fixed slopes $\lambda_1,\lambda_2,\lambda_3,\lambda_4$ the quantity
\[
q(\lambda_1,\lambda_2,\lambda_3,\lambda_4):=\left|
\left(\mathcal{A}_{\lambda_1} + \mathcal{A}_{\lambda_2} \right)
\cap
\left(\mathcal{A}_{\lambda_3} + \mathcal{A}_{\lambda_4} \right)
\right|\,,
\]
and for a fixed bunch $\mathcal{B}$, let 
\[
Q_B := \sum_{\lambda_1,\lambda_2,\lambda_3,\lambda_4\in B} q(\lambda_1,\lambda_2,\lambda_3,\lambda_4)
\]
where the sum is taken over distinct pairs of slopes: $\lambda_1\neq \lambda_2$, $\lambda_3\neq \lambda_4$ and $\{\lambda_1,\lambda_2\} \neq \{\lambda_3,\lambda_4\}$. 

For slopes $\lambda_1,\lambda_2,\lambda_3,\lambda_4 \in B$, with $B\in \mathcal{B}$, the quantity $q(\lambda_1,\lambda_2,\lambda_3,\lambda_4)$ counts (using linear algebra to eliminate the variable $a_4 \in A_{\lambda_4}$) the number of solutions to the equation
\begin{equation}\label{int_s}
a_3= \frac{\lambda_1 -\lambda_4}{\lambda_3-\lambda_4} a_1 + \frac{\lambda_2-\lambda_4}{\lambda_3-\lambda_4} a_2\,,\quad a_1\in A_{\lambda_1}, a_2\in A_{\lambda_2}, a_3 \in A_{\lambda_3}\,.
\end{equation}

We use the inclusion-exclusion principle to count the number of new points in $(A+A)\times (A+A)$ created from forming vector sums from a fixed  bunch $B\in \mathcal{B}$. This number 
is at least 
\begin{equation}\label{e:ie}
\tau^2\binom{N}{2} - Q_B\,.
\end{equation}

If $Q_B\leq N^2\tau^2/8$ for at least half of the bunches $B\in \mathcal{B}$, then \[
|A+A|^2 \geq \frac12\bigg\lfloor\frac{|S_\tau|}{N}\bigg\rfloor \left(\tau^2 \binom{N}{2} - \frac{N^2}{8} \tau^2\right) \geq \frac{|S_\tau|\tau^2 N}{32} \geq \frac{|A|^4N}{128|AA|\log|A|}\,;
\]
by substituting in the value for $N$, we obtain a contradiction. 
Henceforth, let us suppose that $Q_B\geq N^2 \tau^2/4$ for at least $50\% $ of the bunches and let us relabel $\mathcal{B}$ to denote these bunches. 

For $B\in \mathcal{B}$, let 
$\Lambda_B = \{\lambda_3\in B: \sum q(\lambda_1,\lambda_2,\lambda_3,\lambda_4) > \frac{\tau^2 N}{8}\}$, where the sum is taken over slopes $\lambda_1,\lambda_2,\lambda_4\in B$, restricted as in the definition of $Q_B$. The lower bound on $Q_B$ implies that $|\Lambda_B|\geq N/8$. Hence, by an application of the pigeonhole principle on the variables $\lambda_1,\lambda_2,\lambda_4$, we can find a set of at least $N/16$ slopes $\lambda_3 \in B$ so that 
there exist slopes $\lambda_1,\lambda_2,\lambda_4\in B $ such that the number of distinct solutions to equation 
\eqref{int_s}
is 
\begin{equation}\label{e:S}
\geq \frac{\tau^2}{2N^{2}}  :=S\,.
\end{equation}

Let us redefine $S_\tau'$ as the subset of the above slopes $\lambda_3$, so that $|S_\tau'|\geq
\frac{1}{64}|S_\tau|$ -- in fact, we can manipulate the constants in the argument so that $S_\tau'$ constitutes any desired proportion of $S_\tau$.

For each $\lambda\in S_\tau'$, let $A_\lambda'$ be the dyadically popular subset of $A_\lambda$. Namely, we  partition the set of $a \in A_\lambda$ by dyadic values  of their number of realisations as the difference above -- a member of the partition is identified by integer $0\leq s\ll\log|A|$, so that $r_{c_1A_{\lambda_1} - c_2 A_{\lambda_2}} (a)\in [2^s,2^{s+1})$ where $c_1 = (\lambda_1 - \lambda_4)(\lambda - \lambda_4)^{-1}$ and $c_2 = (\lambda_4 - \lambda_1)(\lambda - \lambda_4)^{-1}$. 

Let $A'_\lambda$ be the dyadic subset with the largest contribution to the number of solutions to \eqref{int_s}, by the pigeonhole principle 
$$
\forall a \in A_\lambda', \;\;\;r_{c_1A_{\lambda_1} - c_2 A_{\lambda_2}} (a) \geq \frac{S}{2|A'_\lambda|\log|A|}\,.$$

Applying Lemma \ref{l:fpms} to $A_\lambda'$ with $T= \frac{S}{2|A'_\lambda|\log|A|}$ and $|\Pi_1|,|\Pi_2|\in [\tau,2\tau)$ yields
\begin{equation}\label{e:eal}
\E^\times(A_\lambda') \ll \frac{ \tau^6 |A'_\lambda|^4 (\log|A|)^5 }{S^4}\,.
\end{equation}
After two more applications of the Cauchy-Schwarz inequality, we get
\[
|AA_\lambda|\geq|AA_\lambda'| \geq \frac{|A|^2|A_\lambda'|^2}{\E^\times(A,A'_\lambda)} \geq \frac{|A|^2|A_\lambda'|^2}{\sqrt{\E^\times(A)\E^\times(A'_\lambda)}}\,.
\]
Substituting the value of $S$ from \eqref{e:S} and $\E^\times(A)\gg |S_\tau|\tau^2\log^{-1}|A|$ completes the proof. 

\end{proof}

\section{Proof of Theorem~\ref{t:sumprod}}\label{sec:proofofsumprod}

Without loss of generality, assume that $A$ has positive elements. Denote $|A+A|=K|A|$ and $|AA|=M|A|$. We will show that $\max(K,M)\geq |A|^{\frac1{3}+\exponent-o(1)}$.

Consider the point set $A\times A \in \R^2$. We begin by an application of the dyadic pigeonhole principle to extract a subset of $A\times A$ which supports at least a logarithmic factor of the energy of $A$: decompose the set of $|A/A|$ slopes through the origin supporting $A\times A$ into $\ll \log|A|$ dyadic sets, where each slope in the $j$-th dyadic set $S_j$ contains between $2^{j-1}$ and $2^j-1$ points of $A\times A$. Then $\sum_j|S_j|2^{2(j-1)}\leq \E^\times(A)< \sum_j|S_j|2^{2j}$. Let $j_0$ denote the index for which $|S_{j_0}|2^{2j_0}$ is maximal, and write $S_\tau = S_{j_0}, \tau = 2^{j_0}$. Then $|S_\tau|\tau^2 \gg \E^\times(A)/\log|A|$. We may further assume that $\tau > C$ for some (large) constant $C>0$, since otherwise $\E^\times(A)\leq C^2 |A|^2$, and there is nothing to prove.

Let the points of $A\times A$ on the line with the slope $\lambda \in S_\tau$ be written as 
\[
\mathcal{A}_\lambda:=\{(a,\lambda a)\in A\times A:\,a\in A_\lambda \subseteq A\}\,,\;\;|A_\lambda|\in [\tau,2\tau)\}\,
.\]
The set $A_\lambda=A\cap \lambda^{-1} A$ is the set of the $x$-coordinates of points in $\mathcal{A_\lambda}\subset \ell_\lambda$, where $ \ell_\lambda$ is the line through the origin with slope $\lambda$.

We apply Proposition~\ref{p:main} to $A$ to pass from $\lambda\in S_\tau$  to $\lambda\in S_\tau'$.  Clearly,
$r_{A/A_\lambda}(\lambda)=|A_\lambda|$, by definition of $A_\lambda$.

Thus for any $a\in A$ and any $a_\lambda\in A_\lambda$, one has $$\lambda =\frac{a(\lambda a_\lambda)}{aa_\lambda}\in AA/AA\,.$$
(Konyagin and Shkredov refer to this truism as the Katz and Koester inclusion introduced in \cite{KK}).
There are $|A_\lambda A|$ distinct values of the denominator $aa_\lambda$. Thus, by the lower bound of Proposition \ref{p:main},
\begin{equation}\label{e:T}
\forall \lambda\in S_\tau',\;\;	r_{AA/AA}(\lambda) \;\geq \;|A A_\lambda | \gg  \frac{|A|^6}{|S_\tau|^{1/2}M^4K^8(\log |A|)^7}:= T\,.
\end{equation}

It remains to relate the set $S_\tau'$ to $A$ and use the Theorem \ref{t:main}. We choose a subset of  $S_\tau'$  by intersecting the set of at least $\tau |S_\tau'|$  points of $A\times A$ supported on lines through the origin with slopes in $S_\tau'$,  by a vertical line with some fixed $x$-coordinate $a_0\in A$. By the pigeonhole principle, there is an element $a_0\in A$, where the intersection has cardinality at least the average $\frac{\tau |S_\tau'|}{|A|}$. 

Without loss of generality $a_0=1$. Thus we have $B\subseteq A$, with $|B| \geq \frac{\tau |S_\tau|}{|A|}$ and such that $\forall b\in B,\,r_{AA/AA}(b) \gg T$. Concretely, $B = A/a_0 \cap S_\tau'$.
 
Using $K|A|=|A+A|\geq |B+B|$ we apply Theorem~\ref{t:main} to the set $B$ with $\Pi = AA$ and $T = T$, defined in \eqref{e:T}, to get a lower bound for $|B+B|$. We group together $|S_\tau|\tau^2 \gg \E^\times(A)\log^{-1}|A|$ to take advantage of the Cauchy-Schwarz relation $\E^\times(A)\geq |A|^3/M$, so that (suppressing powers of $\log|A|$):
\[
K^{283}M^{176}\gtrsim |A|^{94}(|S_\tau|\tau^2)^{41/2} |S_\tau|^4\,.
\]
 
We recycle the bound on $T$ to bound $|S_\tau|$: since $M|A|=|AA|\geq |AA_\lambda|\gg T$, we have that $|S_\tau|\geq |A|^{10}M^{-10}K^{-16}$. This yields the inequality $$K^{694}M^{473}\gtrsim |A|^{391}\,$$ whence the claim of Theorem~\ref{t:sumprod} follows.

$\hfill \Box$

\medskip
Reviewing the proof of Theorem \ref{t:sumprod}, we remark that the main reasons why, within the Konyagin-Shkredov strategy, the gain over $\delta=\frac{1}{3}$ is so small are (i) a large power $N^8$ of $N$ (defined by \eqref{e:en} precisely to measure the eventual gain over $\delta=\frac{1}{3}$) in estimate \eqref{e:eal} (where $S= \tau^2N^{-2}$) and (ii) the relative weakness of the forthcoming Theorem \ref{t:main}. Lowering the power $N^8$ would require an unlikely improvement of the symmetric version of Lemma \ref{l:fpms}. A  stronger version of  Theorem \ref{t:main}  might  come from further quantitative progress in understanding the {\em Few Products, Many Sums}, alias weak \ES conjecture, thus re-emphasising its pivotal role in the sum-product theory at large.

\section{Proof of Theorems \ref{t:convex}~and~\ref{t:main}}\label{sec:proofofthm2}
We begin with the following regularisation lemma.

\begin{lemma}\label{lem:regu}
Let $\Reps$ be a deterministic rule (procedure) with parameter $\eps \in (0,1)$ that, to every sufficiently large finite additive set $X$, associates a subset $\Reps(X)\subseteq X$ of cardinality $|\Reps(X)|\geq (1-\eps)|X|$.

For any such rule $\Reps$, any $s>1$ and a sufficiently large finite set $A$,  set $c_1 = \eps \log(|A|) \in (0,1)$. 
Then there exists a set $B\subseteq A$ (depending on $\Reps, \,s$), with $|B|\geq (1-c_1) |A|$ such that 
\[\E_s(\Reps(B))\geq c_2\,\E_s(B)\,,\]
where $c_2 = (1- 2\eps e^{1-s} \in (0,1]$. 
 \end{lemma}
\begin{proof}
Construct a sequence of subsets of $A$ as follows. Set $A_0 := A$; if $\E_s(A_{i+1}) < c_2~\E_s(A_i)$, define $A_{i+1}:=\Reps(A_i)$. By definition of $\Reps$, we have the lower bound $|A_i|\geq (1-\eps)^i|A|\geq (1-i\eps)|A|$ for any index $i$ for which $A_i$ is defined.
 
This process must terminate after at most $I = \lfloor \log|A| \rfloor$ iterations. Indeed, suppose that we have constructed the set $A_I$. Then, using the trivial bound $|A_I|^2\leq \E_s(A_I)\leq |A_I|^{s+1}$, we have 
\[
(1-c_1)^2|A|^2\leq |A_I|^2 \leq \E_s(A_I)< c_2\E_s(A_{I-1}) \leq c_2^I\E_s(A) \leq |A|^{s+1-\log c^{-1}_2}\,.
\]
The choice of $c_2$ yields a contradiction.
\end{proof}
%

\medskip

In this section we prove the following statement implying both theorems.

\begin{theorem} \label{t:five}
    Let finite $A,\Pi_1,\Pi_2\subset \R\backslash\{0\}$ satisfy $|\Pi_1|, |\Pi_2|\geq |A|$.
    
(i) If  $r_{\Pi_1\Pi_2}(a)\geq T$ for all $a\in A$, for some $T\geq 1$,     
    then 
    \[
	    |A+A|^{19} |\Pi_1|^{22}|\Pi_2|^{22} \gg |A|^{41} T^{33} (\log|A|)^{-23}\,.
    \]
    
(ii)  If $A$ is a convex set, then
   \[
   |A+A|\gg |A|^{30/19} (\log|A|)^{-23/19}\,.
   \]
\end{theorem}
    \begin{proof}
We begin with a regularisation argument, applying Lemma~\ref{lem:regu} to the forthcoming procedure $\mathcal{R}$. This will yield a set of $B\subseteq A$ with $|B|\gg |A|$, containing a large subset $\mathcal{R}(B)$ supporting most of the additive energy of $B$. We will then study linear relations among elements of $B$ and $\mathcal{R}(B)$.

Let $P_\eps(A)$ be the set of popular sums of $A$, where `popular' is defined according to some $\eps\in (0,1)$:
\begin{equation}
P_\eps(A):=\left\{x\in A+A:\,r_{A+A}(x) \geq \epsilon \frac{|A|^2}{|A+A|}\right\}\,.
\label{e:pops}
\end{equation}
The set $P_\eps(A)$ supports most of the mass of $A\times A$. That is,
\begin{equation*}
|\{(a,b) \in A\times A:\,a+b\in P_\eps(A)\}|\geq (1-\epsilon)|A|^2\,.
\label{e:pmass}
\end{equation*}
Indeed, we have
\[
|A|^2 = \sum_xr_{A+A}(x) = \sum_{x\in P_\eps(X)}r_{A+A}(x) + \sum_{x\notin P_\eps(X)}r_{A+A}(x) < \sum_{x\in P}r_{A+A}(x) + \eps|A|^2\,.
\]
Let $\mathcal{R}(A)\subseteq A$ correspond to `rich' $x$-coordinates in $A\times A$, namely
\begin{equation*}
\label{e:popabs}
\mathcal{R}(A):=\left\{a\in A\colon |(a+A)\cap P_\eps(A)| \geq \frac{1}{2}|A|\right\}\,.
\end{equation*}
Clearly $\mathcal{R}$ is a deterministic procedure that creates a subset of $A$. We show that $|\mathcal{R}(A)|\geq (1-2\eps)|A|$, or equivalently, in the notation of Lemma~\ref{lem:regu}, that $\mathcal{R}(A)=\mathcal{R}_{2\eps}(A)$. To justify this claim, we use \eqref{e:pops}:
\begin{align*}
(1-\eps)|A|^2 & \leq |\{(a,b)\in \mathcal{R}(A)\times A\colon a+b\in \mathcal{P}_\eps(A)\}| + |\{(a,b)\in (A\setminus\mathcal{R}(A))\times A\colon a+b\in \mathcal{P}_\eps(A)\}|\\
&\leq |A||\mathcal{R}(A)| + \frac{1}{2}|A|(|A|-|\mathcal{R}(A)|)\,.
\end{align*}
A rearrangement shows that $|\mathcal{R}(A)|\geq (1-2\eps)|A|$ and so $\mathcal{R}(A) = \mathcal{R}_{2\eps}(A)$.

Having now defined a deterministic rule, let us apply Lemma~\ref{lem:regu} to the set $A$, setting $\eps = \frac{1}{2}\log^{-1}|A|$. We obtain a set $B\subseteq A$ with $|B|\geq |A|/2$ such that $\E(\mathcal{R}(B)) \gg \E(B)$. The advantage of dealing with the sets $B$ and $\mathcal{R}(B)$ versus $A$ and $\mathcal{R}(A)$ is twofold: to each $b\in \mathcal{R}(B)$ we can add at least $\frac{1}{2}|B|$ distinct members of $B$ to obtain a popular sum in $P_\eps(B)$. Secondly, we have ruled out the adverse potential scenario in which the energy of $\mathcal{R}(A)$ is much less than the energy of $A$.

Let $D\subseteq \mathcal{R}(B)-\mathcal{R}(B)$ be the dyadic set supporting the energy of $\E(\mathcal{R}(B))$, so that, for some $\Delta \geq 1$ we have 
\[
\E(B) \ll \E(\mathcal{R}(B)) \ll \Delta^2|D| \log|A|
\]
and for all $d\in D$, we have $r_{\mathcal{R}(B)-\mathcal{R}(B)}(d)\in [\Delta,2\Delta)$. Note that $D$ also supports the energy of $\E(B)$.

Having defined suitably regular sets, we now proceed to obtain the quantitative advantage of Theorem~\ref{t:main}. This arises from studying the following truism:
\begin{equation}\label{e:truism}
	r-s = (b+r)-(b+s),
\end{equation}
where pairs $(r,s)\in \mathcal{R}(B)\times \mathcal{R}(B)$ are popular by energy: $r-s\in D$, and $b\in B$. 

Let us impose the additional condition that $x:=b+r\in P_\eps(B)$, where $P_\eps(B)$ is defined as in \eqref{e:pops}. Since $b\in B$, there are $\gg |D|\Delta|B|$ solutions to \eqref{e:truism}. We will partition solutions to \eqref{e:truism} as 
\[d=x-y:\;d\in D,\, x\in P_\eps(B),\,y\in B+B\] by  the equivalence relation
\[
(r,s,b)\sim (r+t,s+t,b-t),\;t\in \mathbb R\,.
\]
The  number $Q$ of pairs of related triples is bounded from above by
\begin{equation}\label{e:eqr}
\sum_t r^3_{B-B}(t) =  \E_3(B)\,.
\end{equation}
Hence, by the Cauchy-Schwarz inequality, using the popularity of the set $P_\eps(B)$ and the H\"older inequality respectively, we get

\begin{align}
|A|^2|D|^2\Delta^2 &  \ll \E_3(B) |\{x-y=d\colon x\in P_\eps(B)\;,y\in B+B\;,d\in D\}| \label{e:thm2longtop}\\
&\leq\epsilon^{-1} \E_3(B)|B+B||B|^{-2}|\{b+b' -y = d\colon b,b'\in B\;,y\in B+B\;,d\in D\}|\nonumber\\
&= \epsilon^{-1}\E_3(B)|B+B||B|^{-2}\sum_t r_{B-D}(t)r_{(B+B)-B}(t)\nonumber \\
&\leq \epsilon^{-1} \E_3(B) |B+B||B|^{-2}\E_{3/2}^{2/3}(B,D)\E_3^{1/3}(B,B+B)\,. \nonumber
\end{align}

To each instance of $\E_3$  as well as $\E_{3/2}$ we apply Lemma~\ref{lem:estimatinge3}. We only present the case (i), as the case (ii) of a convex $|A|$ uses the estimates of the same lemma, the numerology change being tantamount to $|\Pi_1|=|\Pi_2|=T=|A|.$

Thus applying Lemma~\ref{lem:estimatinge3} and rearranging we obtain
\begin{equation}\label{e:pen}
|D|^{7/6}\Delta^2 \ll \epsilon^{-7/3} |\Pi_1|^{3}|\Pi_2|^{3}|B|^{-3/2}|B+B|^{5/3}T^{-9/2}\,.
\end{equation}

We can assume, again by Lemma~\ref{lem:estimatinge3},  that 
\begin{equation}\label{e:upper_delta}
\Delta\ll \epsilon^{-1} \frac{|\Pi_1|^2|\Pi_2|^2|B|^2}{T^3\E(B)}\,.
\end{equation} 
Indeed, by definition of $D$, there are $\sim \Delta|D|$ solutions to the equation
\[
r-s = d:\;\;\;r,s\in B,\;d\in D\,.
\]
Estimate \eqref{e:upper_delta} follows by compare this with  bound \eqref{e:bdone} from Lemma~\ref{lem:estimatinge3}, with $C=D$ and rearranging, 
using $\E(B) \gg \epsilon |D|\Delta^2$.

We now multiply both sides of \eqref{e:pen} by $\Delta^{1/6}$, using \eqref{e:upper_delta} in the right-hand side.  After that we rearrange, use
$$|D|\Delta^2 \gg \epsilon^{-1} \E(B) \geq \epsilon^{-1}\frac{|B|^4}{|B+B|},\,$$ as well as  $|A+A|\geq |B+B|$ and $|B|\gg |A|$ to complete the proof.

\end{proof}

We remark that we can easily re-purpose the above proof to retrieve the best known {\em few products, many sums} inequality 
$$|AA|^{14}|A+A|^{10}\geq |A|^{30-o(1)}$$
by Olmezov, Semchankau and Shkredov \cite{OSS}.

We note also that we can obtain similar results involving the difference set by replacing $k=2$ with $k=12/7$ in Lemma \ref{lem:regu}. This recovers the result of Schoen and Shkredov  \cite{schoenshkredov} that $|A-A|\gtrsim |A|^{8/5}$ for $|A|$ convex.

\section{Proof of Theorem~\ref{t:aaaa}}\label{sec:aaaa}

In this section we prove a new lower bound on $|AA+AA|$. The theorem follows immediately by combining the bounds from the two forthcoming propositions, the first one being an easier version of Proposition \ref{p:main} and the second of the argument in the proof of Theorem \ref{t:five} around estimates \eqref{e:truism}-\eqref{e:thm2longtop}. 

We once again assume that $A\subset \R_{>0}$. Similar to Proposition \ref{p:main}, the forthcoming Proposition~\ref{p:aaaa} uses Konyagin and Shkredov's extension  of Solymosi's geometric argument \cite{KS1, KS2, Solymosi}. Only now the vector sums constructed lie 
in  $(AA+AA)\times (AA+AA)$ and the set of slopes used instead of $S_\tau$ are all of the slopes from $A/A$. This idea is due to Balog \cite{Balog}.

A variant of Proposition~\ref{p:aaaa} can be extracted from the paper by  Iosevich, Roche-Newton and the first author \cite[Proof of Theorem 2]{IRR}. Below we give a brief self-contained proof.
\begin{proposition}\label{p:aaaa}
For a finite positive set of positive reals $A$,
\[
|AA+AA|^2\gg  |A/A|^{2/3}|A|^{5/2}\,.
\]
\end{proposition}

\begin{proposition}\label{p:a(a+a)}
Let $A\subseteq \C$. Then
\[
|AA+AA|^5 \gg \frac{|A|^{13}}{|A/A|^5}\log^{-9/2}|A|\,.
\]
\end{proposition}

It remains to prove Propositions~\ref{p:aaaa} and \ref{p:a(a+a)}.

\begin{proof}[Proof of Proposition~\ref{p:aaaa}]
Recall that $A$ is positive.
For each $\lambda\in A/A$ we fix some vector $v_\lambda=(a_\lambda,\lambda a_\lambda)\in A\times A$ lying on the line through the origin with slope $\lambda$. Clearly, for any $b\in A$, the dilates $bv_\lambda$ of the vector $v_\lambda$ are in $AA\times AA$. Thus for any $\lambda_1,\lambda_2\in A/A$, we have the sums of such dilates lie in $(AA+AA)\times (AA+AA)$:
\[
\forall a_1,a_2\in A,\, \lambda_1,\lambda_2 \in A/A,\,  a_1 v_{\lambda_1}+  a_2 v_{\lambda_2} \in (AA+AA)\times (AA+AA)\,.
\]
For fixed distinct $\lambda_1,\lambda_2\in A/A$, we get $|A|^2$  new vector sums, with slope between $\lambda_1$ and $\lambda_2$.

By considering only vector sums from consecutive slopes $\lambda_i,\lambda_{i+1}$, it thus follows that 
\[
|AA+AA|^2\geq (|A/A| - 1) |A|^2\,;\]
we will further attempt to improve by considering vector sums constructed within bunches of slopes.

Similar to  \eqref{e:en} in the proof of Proposition \ref{p:main} define 
\[
N:= C \frac{|AA+AA|^2}{|A|^2|A/A|}\,,
\]
for a sufficiently large absolute $C$.

As in the proof of Proposition \ref{p:main}, partition the set of slopes $A/A$ (equivalently, the lines through the origin supporting $A\times A$)  into $\lfloor |A/A|/N\rfloor$ consecutive ``full'' bunches containing $N$ lines, and at most one bunch consisting of fewer than $N$ lines, which gets deleted. Once again, $N$ is much bigger than $2$ and much smaller than  $|A/A|$.

To each bunch $\mathcal B$ and distinct $\lambda_i,\lambda_j\in \mathcal B$ we construct $|A|^2$ vector sums in $(AA+AA)^2$, by considering the vector sums of the dilates of the vectors $v_{\lambda_i}, v_{\lambda_j}$ by elements of $A$, to generate the sum set $Av_{\lambda_i} + Av_{\lambda_j}\subseteq (AA+AA)^2$. 

Define $q$
and
\[
Q_B = \sum_{\lambda_1,\lambda_2,\lambda_3,\lambda_4 \in B}\left| (Av_{\lambda_i} + Av_{\lambda_j})\cap
(Av_{\lambda_k} + Av_{\lambda_l})\right|
\]
where the sum is taken over $\lambda_1,\lambda_2,\lambda_3,\lambda_4 \in B$ where $\lambda_1\neq \lambda_2$, $\lambda_3\neq \lambda_4$ and $\{\lambda_1,\lambda_2\} \neq \{\lambda_3,\lambda_4\}$/

 By inclusion-exclusion, the number of new elements of $(AA+AA)\times (AA+AA)$ generated by $\mathcal B$ is at least:
\[
\binom{N}{2}|A|^2 - Q_B\,.
\]
It is is a direct analogue of \eqref{e:ie}.
  
By repeating verbatim the argument in the proof of Proposition \ref{p:main} between \eqref{e:ie} and \eqref{e:S}, we argue that, owing to the above choice of $N$, at least for $50\%$ of the bunches $\mathcal B$ the collision term should be large. That is, the value $Q$ associated to at least half the bunches satisfies
\[
Q\gg N^4|A|^2 \left( \frac{|A|^2|A/A|}{|AA+AA|^2} \right)^2\,.
\]
Furthermore, as a direct analogue of the arguments between statements \eqref{int_s} - \eqref{e:S} in the proof of Proposition \ref{p:main} we conclude that there are two dilates of $A$ by some $c_1,\,c_2$, so that
\[
\left| \{ (a,a_1,a_2)\in A^3:\, a=c_1a_1-c_2a_2\}\right| \gg |A|^2N^{-2}\gg  |A|^2  \frac{|A|^4|A/A|^2}{|AA+AA|^4} \,.
\]
We compare this with the upper bound for the number of solutions, from Lemma \ref{lem:estimatinge3}; we use bound \eqref{e:bdone}, with $\Pi_1=A$, $\Pi_2=A/A$ and $T=A$, since $a=b\frac{a}{b}$ for any $b\in A$.

Rearranging completes the proof.
\end{proof}

\begin{proof}[Proof of Proposition~\ref{p:a(a+a)}]
Once again,  dyadically partition the set $A-A$ to obtain $D\subseteq A-A$ and $\Delta\geq 1$ so that $\E(A)\geq |D|\Delta^2\log|A|$ with $r_{A-A}(d)\in [\Delta,2\Delta)$ for each $d\in D$.

There are at least $|A||D|\Delta$ solutions $(a,b,c)\in A^3$ to the equation
\[
a-c = (a+b)-(b+c):\, a-c =d\in D\,.
\]

We proceed as in the proof or Theorem~\ref{t:five}, associating an equivalence relation to these solutions, so that that $(a,b,c)\sim (a+t,b-t,c+t)$ for some $t\in \R$.

We have the analogue of \eqref{e:thm2longtop}
\begin{equation}\label{e:A(A+A)1}
|A||D|\Delta \leq \E_3^{1/2}(A)|\{(x,y,d)\in (A+A)^2\times D\colon x-y=d\}|^{1/2}\,.
\end{equation}
We bound the quantity $|\{x-y=d\colon x,y\in A+A, d\in D\}|$ using Lemma \ref{lem:estimatinge3}, estimate \eqref{e:bdone}, with $\Pi_1=1/A$, $\Pi_2=A(A+A)$ and $T=|A|$, for 
$x= (ax)\frac{1}{a}$ for any $a\in A$. 
 It follows that 
$$
|\{x-y=d\colon x,y\in A+A, d\in D\}| 
\ll |A|^{-1/3}|A+A|^{2/3}|D|^{2/3}|A(A+A)|^{2/3}\,.
$$
Furthermore, once again by Lemma \ref{lem:estimatinge3}, bound \eqref{e:e3(b,x)}, with $\Pi_1=A$, $\Pi_2=A/A$, and $T=|A|$, since $a=b\frac{a}{b}$, for any $b\in A$, we have

$$
\E_3(A)\ll |A/A|^2|A|\log|A|\,.$$ 

Multiplying both sides by $\Delta^{1/3}$, the bound \eqref{e:A(A+A)1} then becomes
\[
|A|^{2/3}(|D|\Delta^2)^{2/3}\ll |A/A||A+A|^{1/3}|A(A+A)|^{1/3}\Delta^{1/3} \log^{1/2}|A|\,.
\]

Similar to \eqref{e:upper_delta}
we have
\[\Delta\ll \frac{|A||A/A|^2}{\E(A)}\log|A| \,,\]

Thus
$$
|A|^{1/3} \E(A) \ll |A/A|^{5/3} |A+A|^{1/3}|A(A+A)|^{1/3} \log^{3/2}|A|\,.
$$

The proof is complete after rearranging after using $\E(A)\geq \frac{|A|^4}{|A+A|}$ and dominating $A+A$ and $A(A+A)$ by $AA+AA$. 
\end{proof}

We remark, curiously, that if one uses an additional assumption in Theorem \ref{t:aaaa} that $A$ is convex, then its estimate improves slightly to
$|AA+AA|\geq |A|^{8/5-o(1)}$. The same exponent $\frac{8}{5}-o(1)$ is the best one known in the few products, many sums scenario for $|A+A|$ when $|AA|\to |A|$, \cite{shkredov}. The same exponent is also the best one known for $|A-A|$ when $A$ is convex (but not for $|A+A|$), \cite{schoenshkredov}. See the discussion in the outset of this paper.

\section*{Acknowledgements}
The authors thank Kit Battarbee for carefully reading the draft of this manuscript and Ilya Shkredov for discussions.

\Addresses

\end{document}